\documentclass[12pt,oneside]{amsart}
\usepackage[utf8]{inputenc}
\usepackage{amsmath}
\usepackage{tikz-cd}
\usepackage{bm}
\usepackage{amssymb}
\usepackage{mathtools}
\usepackage{amsthm}
\usepackage{graphicx}
\usepackage{comment}
\usepackage{comment}
\usepackage{appendix}
\usepackage{caption}
\usepackage{subcaption}
\usepackage{stackrel}
\usepackage{hyperref}
\usepackage{upgreek}
\usepackage{enumitem}
\usepackage{comment}
\usepackage{tikz}
\usepackage{pict2e,mathrsfs}

\makeatletter
\let\@wraptoccontribs\wraptoccontribs
\makeatother

\title[Approximability and Reflexivity]{Approximable Triangulated categories and Reflexive DG-categories}

\author{Isambard Goodbody}
\address{Isambard Goodbody, School of Mathematics and Statistics, University of Glasgow, University
Place, Glasgow G12 8QQ}

\contrib[with an appendix by]{Theo Raedschelders and Greg Stevenson}

\DeclareMathOperator{\rad}{rad}

\DeclareMathOperator{\Hom}{Hom}

\DeclareMathOperator{\RHom}{RHom}

\DeclareMathOperator{\thick}{thick}

\DeclareMathOperator{\Mod}{Mod}

\DeclareMathOperator{\fdmod}{mod}
\DeclareMathOperator{\perf}{perf}

\DeclareMathOperator{\Ext}{Ext}

\DeclareMathOperator{\fd}{fvd}
\DeclareMathOperator{\pvd}{pvd}

\DeclareMathOperator{\Qcoh}{Qcoh}
\DeclareMathOperator{\coh}{coh}
\DeclareMathOperator{\Hmo}{Hmo}

\begin{document}

\usetikzlibrary{matrix}
\usetikzlibrary{shapes}

\theoremstyle{plain}
\newtheorem{prop}{Proposition}[section]
\newtheorem{lemma}[prop]{Lemma}
\newtheorem{theorem}[prop]{Theorem}
\newtheorem{cor}[prop]{Corollary}
\newtheorem{conj}[prop]{Conjecture}

\theoremstyle{definition}
\newtheorem{defn}[prop]{Definition}
\newtheorem{ass}[prop]{Assumption}
\newtheorem{cons}[prop]{Construction}
\newtheorem{ex}[prop]{Example}
\newtheorem{remark}[prop]{Remark}
\newtheorem*{ack}{Acknowledgements}
\maketitle

\tikzset{
    vert/.style={anchor=south, rotate=90, inner sep=.5mm}
}

\newcommand{\rightarrowdbl}{\rightarrow\mathrel{\mkern-14mu}\rightarrow}

\newcommand{\xrightarrowdbl}[2][]{%
  \xrightarrow[#1]{#2}\mathrel{\mkern-14mu}\rightarrow
}

\begin{abstract}
We use the theory of approximable triangulated categories to give a condition for a proper DG-category to be reflexive in the sense of Kuznetsov and Shinder. To do this we provide another description of the completion of an approximable triangulated category under a properness assumption. We apply our results to proper schemes, proper connective DG-algebras and Azumaya algebras over proper schemes. We include an appendix by Raedschelders and Stevenson showing that proper connective DG-algebras admit finite dimensional models over any field.
\end{abstract}

\tableofcontents

\section{Introduction}

For a projective scheme $X$, there are two small derived categories one traditionally works with: the perfect complexes $\mathcal{D}^{\perf}(X)$ and the bounded derived category of coherent sheaves $\mathcal{D}^b(\coh X)$. Reflexive DG-categories and approximable triangulated categories are both tools for describing the relationship between these categories. The aim of this paper is to compare these two approaches. 
\\

The theory of approximable triangulated categories makes precise the idea that $\mathcal{D}^b(\coh X)$ is the completion of $\mathcal{D}^{\perf}(X)$ with respect to a metric (and vice versa). Whereas the the theory of reflexive DG-categories makes precise the idea that $\mathcal{D}^b(\coh X)$ is the dual of $\mathcal{D}^{\perf}(X)$ (and vice versa). Under a finiteness assumption, Neeman showed that approximable triangulated categories satisfy some representability theorems which are reminiscint of the definition of a reflexive DG-category. We clarify this connection and give a condition for a DG-category to be reflexive assuming approximability and a strong generation condition.
\\

As applications we show that proper schemes and Azumaya algebras over proper schemes are reflexive over any regular ring and we show that proper connective DG-algebras are reflexive over any field. We provide two proofs that the strong generation condition is satisfied in the latter example. One is direct and one is based on taking finite dimensional models using a theorem of Raedschelders and Stevenson. This generalises the key examples of Kuznetsov and Shinder. Unless stated otherwise, $k$ will denote a commutative Noetherian ring and $\fdmod k$ will denote the finitely generated $k$-modules.

\begin{ack}
I'm grateful to Timothy De Deyn, Fei Peng, Kabeer Manali Rahul,  and Greg Stevenson for valuable comments. I'd like to thank an anonymous reviewer for pointing out a simplification of the proof of Lemma 4.1 and Timothy De Deyn for pointing out a mistake in a previous version of Lemma 5.2. 
\end{ack}

\section{Approximable Triangulated categories}

Developed by Neeman, the theory of approximable triangulated categories has found many applications which 
are surveyed in \cite{Nee21a} and \cite{CNS24}. The main objects of study here are triangulated categories with coproducts and $t$-structures. In order to talk about approximable triangulated categories we need some notation on how to build objects. 

\begin{defn} Let $\mathcal{T}$ be a triangulated category with coproducts, $S \subseteq \mathcal{T}$ a subcategory, and let $[a,b] \subseteq \mathbb{Z}$ be an interval. 
\begin{enumerate}
\item Let $\langle S \rangle_1^{[a,b]} \subseteq \mathcal{T}$ consist of finite sums of $i$-th shifts of summands of objects of $S$ where $i \in [a,b]$. 

\item Let $\langle S \rangle_n^{[a,b]} \subseteq \mathcal{T}$ consist of all summands of objects $X$ which fit into a triangle
\[
X' \to X \to X'' \to^+
\]
with $X' \in \langle S \rangle_1^{[a,b]}$ and $X'' \in \langle S \rangle_{n-1}^{[a,b]}$.

\item Let $\langle S \rangle^{[a,b]}$ be the union over $n$ of all ${\langle S \rangle}^{[a,b]}_n$.

\item Let $\overline{\langle S \rangle_1}^{[a,b]}$ consist of all sums of $i$-th shifts of summands of objects of $S$ where $i \in [a,b]$.

\item Let $\overline{\langle S \rangle}_n^{[a,b]}$ consist of all summands of objects $X$ which fit into a triangle
\[
X' \to X \to X'' \to^+
\]
with $X' \in \overline{\langle S \rangle}_1^{[a,b]}$ and $X'' \in \overline{\langle S \rangle}_{n-1}^{[a,b]}$.

\end{enumerate}

If $[a,b] = \mathbb{Z}$ so that all shifts are allowed, we will omit the associated superscript. 

\end{defn}

\begin{defn} \label{approxdefn}
A triangulated category is approximable if it has a compact generator $G$ and a $t$-structure $(\mathcal{T}^{\leq 0}, \mathcal{T}^{\geq 0})$ and there is an integer $A >0$ such that 
\begin{enumerate}
\item $\Sigma^A G \in \mathcal{T}^{\leq 0}$
\item $\Hom_{\mathcal{T}}(\Sigma^{-A} G, E) = 0$ for all $E \in \mathcal{T}^{\leq 0}$
\item For every $F \in \mathcal{T}^{\leq 0}$ there is a triangle
\[
E \to F \to D \to^+
\]
with $D \in \mathcal{T}^{\leq -1}$ and $E \in \overline{\langle G \rangle}^{[-A,A]}_A$
\end{enumerate}
\end{defn}

\begin{remark}
By Proposition 2.6 in \cite{Nee21c}, approximability is a property only of the triangulated category and don't depend on a choice of $t$-structure or compact generator.
\end{remark}

\begin{defn} \label{tstructdef} Given a $t$-structure $(\mathcal{T}^{\leq 0}, \mathcal{T}^{\geq 0})$ on a triangulated category $\mathcal{T}$, we fix the following notation
\[
\mathcal{T}^{\leq n} := \Sigma^{-n} \mathcal{T}^{\leq 0}, \quad \mathcal{T}^{\geq n} := \Sigma^{-n} \mathcal{T}^{\geq 0}
\]
\[
\mathcal{T}^- := \bigcup_n \mathcal{T}^{\leq n}, \quad \mathcal{T}^+ := \bigcup_n \mathcal{T}^{\geq n},\quad \mathcal{T}^b := \mathcal{T}^- \cap \mathcal{T}^+
\]
Two $t$-structures $(\mathcal{T}^{\leq 0}_1,\mathcal{T}^{\geq 0}_1)$ and $(\mathcal{T}^{\leq 0}_2,\mathcal{T}^{\geq 0}_2)$ on $\mathcal{T}$ are equivalent if there is some $N \geq 0$ such that $\mathcal{T}^{\leq -N}_1 \subseteq \mathcal{T}^{\leq 0}_2 \subseteq \mathcal{T}^{\leq N}_1$. 
\end{defn}

\begin{remark}
By Theorem A.1 in \cite{AJS03}, any compact object in a triangulated category $\mathcal{T}$ with coproducts generates a $t$-structure. One can check that if $G$ and $G'$ are two compact generators of $\mathcal{T}$, then they generate equivalent $t$-structures. A $t$-structure is said to be in the preferred equivalence class if it is equivalent to one generated by a compact generator of $\mathcal{T}$. By Proposition 2.4 in \cite{Nee21c}, any $t$-structure on an approximable triangulated category $\mathcal{T}$ satisfying the conditions (1), (2) and (3) in Definition \ref{approxdefn}, is in the preferred equivalence class. 
\end{remark}

\begin{ex} \label{approxeg} The following examples appear in \cite{CNS24} and \cite{Nee21a}.
\begin{enumerate}
\item If $X$ is a quasi-compact quasi-separated scheme then $\mathcal{D}_{\Qcoh}(X)$ is approximable and the standard $t$-structure is in the preferred equivalence class.

\item If $\mathcal{T}$ is a triangulated category with a compact generator such that $\Hom_{\mathcal{T}}(G,\Sigma^i G) = 0$ for $i > 0$, then $\mathcal{T}$ is approximable. For example connective $E_1$ algebras are approximable and the standard $t$-structures are in the preferred equivalence class.
\end{enumerate}
\end{ex}

The following is the non-commutative version of $\mathcal{D}^b_{\coh}$ associated to an approximable triangulated category.

\begin{defn} Let $\mathcal{T}$ be an approximable triangulated category and $(\mathcal{T}^{\leq 0}, \mathcal{T}^{\geq 0})$ a $t$-structure in the preferred equivalence class.
\begin{enumerate}

\item Let $\mathcal{T}_c^-$ consist of all objects $F \in \mathcal{T}$ such that for any $n >0$ there is a triangle
\[
E \to F \to D \to^+
\]
with $E \in \mathcal{T}^c$ and $D \in  \mathcal{T}^{\leq -n}$.

\item Let $\mathcal{T}^b_c = \mathcal{T}^{-}_c \cap \mathcal{T}^b$.
\end{enumerate}
\end{defn}

\begin{remark}
The subcategories $\mathcal{T}_c^{-}$ and $\mathcal{T}^b_c$ are thick and only depend on $\mathcal{T}$ and not on a choice of $t$-structure in the preferred equivalence class. See Lemma 2.10 in \cite{Nee21c} and Remark 8.7 in \cite{Nee21a}.
\end{remark}

\begin{ex} \noindent
\begin{enumerate}
\item For a quasi-compact quasi-separated scheme $X$, $\mathcal{D}_{\Qcoh}(X)^b_c$ is the category of pseudocoherent complexes with bounded cohomology. So if $X$ is Noetherian $\mathcal{D}_{\Qcoh}(X)^b_c = \mathcal{D}^b_{\coh}(X)$. See the discussion after Proposition 8.10 in \cite{Nee21a}.

\item Suppose $A$ is a connective $E_1$-algebra such that $\pi_0(A)$ is a coherent ring and $\pi_n(A)$ is a finitely presented $\pi_0(A)$ module for all $n \in \mathbb{Z}$. Theorem A.5 in \cite{BCRPZ24} shows that $\mathcal{D}(A)^b_c$ consists of all modules $M$ such that $\pi_n(M)$ is finitely presented over $\pi_0(A)$ for all $n \in \mathbb{Z}$ and only finitely many are non-zero.
\end{enumerate}
\end{ex}

We will use the following representability theorems which hold for approximable triangulated categories. Recall $k$ is a commutative Noetherian ring.

\begin{defn} 
For a $k$-linear triangulated category $\mathcal{T}$, a finite homological functor is a $k$-linear homological functor $F: \mathcal{T} \to \Mod k$ such that $\bigoplus_i F(\Sigma^i t) \in \fdmod k$ for every $t \in \mathcal{T}$.
\end{defn}

\begin{theorem}[Theorem 7.20 in \cite{Nee21c} and Theorem 4.7 in \cite{Nee18b}] \label{ApproxRep}
Let $\mathcal{T}$ be an approximable triangulated category such that $\Hom_{\mathcal{T}}(X,Y)$ is a finitely generated $k$-module for all $X,Y \in \mathcal{T}^c$.

\begin{enumerate}
\item 
A $k$-linear homological functor $F: (\mathcal{T}^c)^{op} \to \Mod k$ is finite if and only if it is isomorphic to $\Hom_{\mathcal{T}}(-,M)$ for some $M \in \mathcal{T}^b_c$.

\item Suppose there is some $G \in \mathcal{T}^b_c$ and $N > 0$ such that $\mathcal{T} = \overline{\langle G \rangle}_N$. Then a $k$-linear homological functor $F: \mathcal{T}^b_c \to \Mod k$ is finite if and only if it is isomorphic to $\Hom_{\mathcal{T}}(M,-)$ for some $M \in \mathcal{T}^c$.
\end{enumerate}
\end{theorem}

\begin{remark}
We won't use it here but Theorem 7.20 in \cite{Nee21c} and Theorem 4.7 in \cite{Nee18b} also make statements about representability of morphisms between finite functors on $\mathcal{T}^b_c$ and $\mathcal{T}^c$. 
\end{remark}

\section{Reflexive DG-categories}
We recall some facts about reflexive DG-categories as introduced in \cite{KS22}. By a DG-category we mean a category enriched in chain complexes over a commutative Noetherian ring $k$. See \cite{Kel06} for a survey of DG-categories. If $\mathcal{A}$ is a small DG-category over $k$, we let $\mathcal{D}(\mathcal{A})$ denote its derived category (of right modules), $\mathcal{D}^{\perf}(\mathcal{A})$ the compact objects in $\mathcal{D}(\mathcal{A})$ and $\mathcal{D}_{\fd}(\mathcal{A})$ the subcategory of $\mathcal{D}(\mathcal{A})$ consisting of finitely valued DG-modules $M$ i.e.\ $M(a) \in \mathcal{D}^b(\fdmod k)$ for each $a$. We also let $\mathcal{D}_{\pvd}(\mathcal{A})$ denote the subcategory of $\mathcal{D}(\mathcal{A})$ consisting of perfectly valued DG-modules $M$ i.e.\ $M(a) \in \mathcal{D}^{\perf}(k)$ for each $a$. We view $\mathcal{D}(\mathcal{A})$, $\mathcal{D}^{\perf}(\mathcal{A})$, $\mathcal{D}_{\fd}(\mathcal{A})$ and $\mathcal{D}_{\pvd}(\mathcal{A})$ as DG-categories themselves with their natural enhancements. A DG-category $\mathcal{A}$ is locally finite if $\mathcal{A}(a,b) \in \mathcal{D}^b(\fdmod k)$, and proper if $\mathcal{A}(a,b) \in \mathcal{D}^{\perf}(k)$.

\begin{defn} \label{refldefn}
A small DG-category $\mathcal{A}$ is Morita-reflexive if the map 
\[
\mathcal{D}^{\perf}(\mathcal{A}) \to \mathcal{D}_{\pvd}(\mathcal{D}_{\pvd}(\mathcal{A})^{op})^{op}; M \mapsto \RHom_{\mathcal{A}}(M,-)
\]
is an equivalence. A small DG-category $\mathcal{A}$ is finite-reflexive if the map 
\[
\mathcal{D}^{\perf}(\mathcal{A}) \to \mathcal{D}_{\fd}(\mathcal{D}_{\fd}(\mathcal{A})^{op})^{op}; M \mapsto \RHom_{\mathcal{A}}(M,-)
\]
is an equivalence. 
\end{defn}

\begin{remark}
\begin{enumerate}
\item To justify the name Morita-reflexive we note that in the closed symmetric monoidal category $\Hmo$ of small DG-categories localised at Morita equivalences, the (weak) dual of $\mathcal{A}$ is $\mathcal{D}_{\pvd}(\mathcal{A})$. 

\item If $k$ is regular, then $\mathcal{D}^b(\fdmod k) = \mathcal{D}^{\perf}(k)$ and $\mathcal{A}$ is Morita-reflexive if and only if it is finite-reflexive. In this case we simply say $\mathcal{A}$ is reflexive.

\end{enumerate}
\end{remark}

\begin{ex} \noindent \label{refleg} 
\begin{enumerate}
\item In Lemma 3.14 of \cite{KS22}, it is shown that if $\mathcal{A}$ is a proper DG-category over a field then it is semi-reflexive. i.e.\ the evaluation map in Definition \ref{refldefn} is fully faithful. The same proof shows that if $\mathcal{A}$ is locally finite over a commutative Noetherian ring $k$, then the map 
\[
\mathcal{D}^{\perf}(\mathcal{A}) \to \mathcal{D}_{\fd}(\mathcal{D}_{\fd}(\mathcal{A})^{op})^{op}
\]
is fully faithful.

\item If $X$ is a projective scheme over a perfect field then it is shown in Proposition 6.1 of \cite{KS22}, that $\mathcal{D}^{\perf}(X)$ is reflexive and $\mathcal{D}_{\pvd}(\mathcal{D}^{\perf}(X)) \simeq \mathcal{D}^b_{\coh}(X)$.

\item If $A$ is a proper connective DG-algebra ($H^\ast(A)$ vanishes in positive degrees) over a perfect field then by Proposition 6.9 of \cite{KS22},  $\mathcal{D}^{\perf}(A)$ is reflexive. 

\item In \cite{BNP13}, it is shown that if $X$ is a proper algebraic space over a field of characteristic zero then $\mathcal{D}^{\perf}(X)$ is reflexive and $\mathcal{D}_{\pvd}(\mathcal{D}^{\perf}(X)) \simeq \mathcal{D}^b_{\coh}(X)$. 
\end{enumerate}
\end{ex}

\begin{remark} \noindent For this remark suppose $k$ is a field. Theorem 3.17 in \cite{KS22} states that there is a bijection between the semi-orthogonal decompositions of $\mathcal{D}^{\perf}(\mathcal{A})$ and $\mathcal{D}_{\pvd}(\mathcal{A})$ for a reflexive DG-category $\mathcal{A}$. By Corollary 3.16 in loc.\ cit.\ they have the same triangulated autoequivalence groups. In \cite{Goo24}, it was shown that reflexive DG-categories are the reflexive objects in the closed symmetric monoidal category $\Hmo$ constructed in \cite{toe06}. It follows that $\mathcal{D}^{\perf}(\mathcal{A})$ and $\mathcal{D}_{\pvd}(\mathcal{A})$ have the same Hochschild cohomologies and derived Picard groups.  
\end{remark}

\begin{defn}
We say a small DG-category $\mathcal{A}$ is approximable if  $\mathcal{D}(\mathcal{A})$ is an approximable triangulated category.
\end{defn}

\begin{remark} Since approximable triangulated categories admit a single compact generator, any approximable DG-category is Morita equivalent to a DG-algebra. 
\end{remark}
 
\section{Reflexivity via Approximability}

In this section we identify the completion of an approximable locally finite DG-category and use this to show that finite-reflexivity holds under a strong generation assumption.

\begin{lemma} \label{keylemma}
Suppose $\mathcal{T}$ is a $k$-linear approximable triangulated category and $\Hom_{\mathcal{T}}(X,Y) \in \fdmod k$ for all $X,Y \in \mathcal{T}^c$. Then 
\[
\mathcal{T}^b_c = \left\{t \in \mathcal{T} \mid \bigoplus_{i} \Hom_\mathcal{T}(c,\Sigma^i t) \in \fdmod k \text{ for any } c \in \mathcal{T}^c \right\}
\]
\end{lemma}

\begin{proof}
Let $\mathcal{T}^{rhf}$ denote the right hand side. By Theorem \ref{ApproxRep} (1) (or directly), we see that $\mathcal{T}^b_c \subseteq \mathcal{T}^{rhf}$. Let $G$ be a compact generator of $\mathcal{T}$ and $(\mathcal{T}^{\leq 0},\mathcal{T}^{\geq 0})$ the $t$-structure it generates. Then $(\mathcal{T}^{\leq 0},\mathcal{T}^{\geq 0})$ is in the preferred equivalence class and we may use it to define $\mathcal{T}^b_c$. If $X$ is contained in $\mathcal{T}^{rhf}$, then $\Hom_{\mathcal{T}}(c,\Sigma^i X)$ vanishes for $\lvert i \rvert >> 0$. So by Lemma 3.9 iv) and v) of \cite{BNP}, we see that $X \in \mathcal{T}^b$. 
\\

Now we must check $X \in \mathcal{T}_c^{-}$. By assumption, $\Hom_{\mathcal{T}}(-,X)$ is a finite functor on $\mathcal{T}^c$ and so by Theorem \ref{ApproxRep}, we have a natural isomorphism $\alpha: \Hom_{\mathcal{T}}(-,X') \simeq \Hom_{\mathcal{T}}(-,X)$ of functors on $\mathcal{T}^c$ for some $X' \in \mathcal{T}^b_c$. By Lemma 7.5 ii) in \cite{Nee21c}, $X'$ admits a (strong) $\mathcal{T}^c$-approximating system in the terminology of loc.\ cit.\ and so Lemma 5.8 of loc.\ cit.\ applies and states that $\alpha$ lifts to a morphism $\hat{\alpha}: X' \to X$. Since $\Hom_{\mathcal{T}}(-,\hat{\alpha})$ is an isomorphism restricted to compacts, and $\mathcal{T}$ is compactly generated, it follows that $\Hom_{\mathcal{T}}(-,\hat{\alpha})$ is an isomorphism on all $\mathcal{T}$. By Yoneda, $\hat{\alpha}$ is an isomorphism. 

\end{proof}

\begin{remark}
In the description of $\mathcal{T}^b_c$ in Lemma \ref{keylemma}, it is equivalent to ask for $\bigoplus_i \Hom(c,\Sigma^i t )$ to be finite over $k$ for all $c$ contained in a compact generating set. 
\end{remark}

\begin{theorem} \label{tbcdesc}
Suppose $\mathcal{A}$ is an approximable DG-category such that $H^i\mathcal{A}(a,b) \in \fdmod k$ for all $a,b \in \mathcal{A}$ and $i \in \mathbb{Z}$ (e.g.\ $\mathcal{A}$ is locally finite). Then $\mathcal{D}(\mathcal{A})^b_c = \mathcal{D}_{\fd}(\mathcal{A})$.
\end{theorem}

\begin{proof}
The representables $\mathcal{A}(-,a)$ for $a \in \mathcal{A}$ form a set of compact generators for $\mathcal{D}(\mathcal{A})$. So for any two $M,N \in \mathcal{D}^{\perf}(\mathcal{A})$, $\Hom_{\mathcal{D}(\mathcal{A})}(M,\Sigma^i N) \in \fdmod k$ for all $i$. Since $\Hom_{\mathcal{D}(\mathcal{A})}(\mathcal{A}(-,a),\Sigma^i M) = H^i(M(a))$, the right hand side of the equality in Lemma 4.1 can be identified with $\mathcal{D}_{\fd}(\mathcal{A})$.
\end{proof}

\begin{theorem} \label{refltest}
Suppose $\mathcal{A}$ is a locally finite approximable DG-category. If there is an object $G \in \mathcal{D}_{\fd}(\mathcal{A})$ such that $\mathcal{D}(\mathcal{A}) = \overline{\langle G \rangle}_n$, then $\mathcal{A}$ is finite-reflexive.
\end{theorem}

\begin{proof} 
By Theorem \ref{tbcdesc}, we have that $\mathcal{D}(\mathcal{A})^b_c = \mathcal{D}_{\fd}(\mathcal{A})$. As $\mathcal{A}$ is locally finite, by Example \ref{refleg} (1) it is enough to show the map in Definition \ref{refldefn} is essentially surjective. Suppose that $F \in \mathcal{D}_{\fd}(\mathcal{D}_{\fd}(\mathcal{A})^{op}) = \mathcal{D}_{\fd}(\mathcal{D}(\mathcal{A})^b_c)$. Then $H^0(F): \mathcal{D}(\mathcal{A})^b_c \to \Mod k$ is a homologically finite functor. Theorem \ref{ApproxRep} gives that $H^0(F) \simeq \Hom_{\mathcal{D}(\mathcal{A})}(M,-)$ for some $M \in \mathcal{D}^{\perf}(\mathcal{A})$. It follows (by e.g.\ the dual of  Lemma 2.2 in \cite{KS22}), that $F \simeq \RHom_\mathcal{A}(M,-) \in \mathcal{D}_{\fd}(\mathcal{D}_{\fd}(\mathcal{A})^{op})$. Therefore $F$ is in the image of the map $\mathcal{D}^{\perf}(\mathcal{A}) \to \mathcal{D}_{\fd}(\mathcal{D}_{\fd}(\mathcal{A}^{op}))^{op}$ as required.
\end{proof}

\label{sec3}

\section{Examples}

We apply the results of Section \ref{sec3} to three examples, proper connective DG-algebras, proper schemes and Azumaya algebras over proper schemes.

\begin{cor} \label{cor2}
If $X$ is a proper scheme over $k$, then $\mathcal{D}^b_{\coh}(X) \simeq \mathcal{D}_{\fd}(\mathcal{D}^{\perf}(\mathcal{A}))$ and $\mathcal{D}^{\perf}(X)$ is finite-reflexive.
\end{cor}

\begin{proof}
By Example \ref{approxeg}, $\mathcal{D}_{\Qcoh}(X)$ is approximable and $\mathcal{D}_{\Qcoh}(X)^b_c = \mathcal{D}^b_{\coh}(X)$. Over a field it was shown in \cite{Orl16} that $\mathcal{D}^{\perf}(X)$ is a locally finite DG-category. A similar proof works over a commutative Noetherian ring. It follows for example Proposition \ref{ncproper}. So Theorem \ref{tbcdesc} applies and $\mathcal{D}^b_{\coh}(X) \simeq \mathcal{D}_{\fd}(\mathcal{D}^{\perf}(\mathcal{A}))$.  By Theorem 2.3 of \cite{Nee21b} there is an object $G \in \mathcal{D}^b_{\coh}(X)$ such that $\overline{ \langle G \rangle}_n = \mathcal{D}(\mathcal{D}^{\perf}(X))$. We are done by Theorem \ref{refltest}.
\end{proof}
\begin{lemma} \label{connstrongen}
Suppose $A$ is a proper connective DG-algebra over a field $k$. Then there is an $N \geq 1$ such that $\mathcal{D}(A) = \overline{\langle H^0(A)/\rad H^0(A) \rangle}_N$
\end{lemma}

\begin{proof}
Since $A$ is connective so is $A^e$ and since $A$ is bounded the standard $t$-structure implies there is some $N$ such that 
\[
A \in \langle H^\ast(A)\rangle_N \subseteq \mathcal{D}(A^e)
\]
There is a filtration by $H^0(A)$-bimodules,
\[
0 = H^\ast(A) J^{N'} \subseteq \dots \subseteq H^\ast(A) J \subseteq H^\ast(A)
\]
where $J = \rad H^0(A)$. Let $F_{i} = H^\ast(A)J^i/H^\ast(A)J^{i+1}$ for $i = 0,\dots,{N'-1}$ which are $H^0(A)^{op} \otimes_k H^0(A)/J$-modules. Since $H^0(A)/J$ is semi-simple, each of the $F_{i}$ are semisimple right $H^0(A)$-modules. We have that 
\[
H^\ast(A) \in \langle F_{i} \mid i =  0,\dots,{N'-1} \rangle_{N'} \subseteq \mathcal{D}(H^0(A)^{op} \otimes_k H^0(A)/J)
\]
Combining with the above and restricting along the map $A^e \to H^0(A^e) \cong H^0(A)^e \to H^0(A)^{op} \otimes_k H^0(A)/J$ then gives that $A \in \langle F_{i} \rangle_{N''} \subseteq \mathcal{D}(A^e)$ for some $N'' \geq 1$.  So there is a sequence in $\mathcal{D}(A^e)$,
\[
0 \to G_0 \to G_1 \to \dots \to G_{N''} = A
\]
with the cone of $G_j \to G_{j+1}$ lying in $\langle F_{i}  \rangle_1$. Now if $M \in \mathcal{D}(A)$, tensoring the above diagram with $M$ gives a diagram
\[
0 \to M \otimes_A^{\mathbb{L}} G_0 \to M \otimes_A^{\mathbb{L}} G_1 \to \dots \to M \otimes_A^{\mathbb{L}} G_{N''} = M
\]
with each factor lying in $\langle M \otimes_A^{\mathbb{L}} F_{i} \rangle_1$. Since each $F_{i}$ is a semisimple right $H^0(A)$ module we have that $M \otimes_A^{\mathbb{L}} F_{i} \in \mathcal{D}(H^0(A)/J)$. Therefore $M \otimes_A^{\mathbb{L}} F_{i} \in \overline{\langle H^0(A)/J \rangle}_1$. Therefore $M \in \overline{\langle H^0(A)/J \rangle}_{N''}$.
\end{proof}

\begin{remark}
Lemma \ref{connstrongen} can also be proved by taking a finite dimensional model. In \cite{Orl20}, Orlov constructed the DG-radical $\rad(A)_-$ of a finite dimensional DG-algebra $A$ and it follows from the results in loc. cit.\ that $\mathcal{D}(A) = \overline{\langle A/\rad(A)_- \rangle}_N$ where $\rad(A)_{-}^N = 0$. By Theorem \ref{thm:cyoneda}, any proper connective DG-algebra over a field is quasi-isomorphic to a finite dimensional DG-algebra.
\end{remark}

\begin{prop}
Suppose $A$ is a proper connective DG-algebra over a field $k$ then $A$ is reflexive.
\end{prop}

\begin{proof}
This follows from Theorem \ref{refltest} and Lemma \ref{connstrongen}.
\end{proof}

\begin{defn}
A mild non-commutative scheme $(\mathcal{A},X)$ is a scheme $X$ with a sheaf $\mathcal{A}$ of $\mathcal{O}_X$-algebras such that $\mathcal{A}$ is quasi-coherent when viewed as an $\mathcal{O}_X$-module. We say $(\mathcal{A},X)$ is quasi-compact, quasi-separated or Noetherian, if the same holds for $X$.
\end{defn}

We recall some notation for a quasi-compact quasi-separated mild non-commutative scheme $(\mathcal{A},X)$.  Let $\Mod (\mathcal{A},X)$ denote the category of $\mathcal{A}$-modules,  $\Qcoh(\mathcal{A},X)$ the quasi-coherent $\mathcal{A}$-modules, $\mathcal{D}(\mathcal{A},X)$ the derived category of $\mathcal{A}$-modules and $\mathcal{D}_{\Qcoh}(\mathcal{A},X)$ the derived category of $\mathcal{A}$-modules with quasi-coherent cohomology. The compact objects in $\mathcal{D}_{\Qcoh}(\mathcal{A},X)$ are the perfect complexes which we denote $\mathcal{D}^{\perf}(\mathcal{A},X)$. If $X$ is Noetherian and $\mathcal{A} \in \coh(X)$, then $\coh(\mathcal{A},X)$ will denote the coherent $\mathcal{A}$-modules and $\mathcal{D}^b_{\coh}(\mathcal{A},X)$ will denote the subcategory of $\mathcal{D}_{\Qcoh}(\mathcal{A},X)$ consisting bounded complexes with coherent cohomology. 
For $M,N \in \Mod(\mathcal{A},X)$, there are sheaf hom's $\mathcal{H}om_{\mathcal{A}}(M,N) \in \Mod(\mathcal{O}_X)$ and sheaf Ext's $\mathcal{E}xt^i_{\mathcal{A}}(M,N) \in \Mod(\mathcal{O}_X) $ constructed similarly to those for $\mathcal{O}_X$-modules. For any two $M,N \in \mathcal{D}(\mathcal{A},X)$ there is a derived sheaf hom $\mathcal{R}\mathcal{H}om_{\mathcal{A}}(M,N) \in \mathcal{D}(\mathcal{A},X)$ such that $H^i(X,\mathcal{R}\mathcal{H}om_{\mathcal{A}}(M,N)) = \Hom_{\mathcal{D}(\mathcal{A},X)}(M,\Sigma^i N)$. Also if $M,N \in \Mod \mathcal{O}_X$, $H^i(\mathcal{R}\mathcal{H}om_{\mathcal{A}}(M,N)) = \mathcal{E}xt^i_{\mathcal{A}}(M,N)$.

\begin{lemma} \label{ncproplem1}
Suppose $(\mathcal{A},X)$ is a mild noncommutative scheme with $X$ proper over $k$ and $\mathcal{A} \in \coh(X)$. Then for any $M,N \in \mathcal{D}^b_{\coh}(\mathcal{A})$, $\Hom_{\mathcal{D}^b_{\coh}(\mathcal{A})}(M,\Sigma^i N) \in \fdmod k$. 
\end{lemma}

\begin{proof}
By induction on the length of the complexes and using the standard $t$-structure we can reduce to the case of $M,N \in \coh(\mathcal{A})$. By a similar proof to the commutative case we have that $\mathcal{E}xt^i_{\mathcal{A}}(M,N) \in \coh(X)$ for all $i$. Hence $H^p(X,\mathcal{E}xt^i_{\mathcal{A}}(M,N)) \in \fdmod k$ (see for example \cite[Tag 02O5]{stacks} ). There is a degenerate spectral sequence (Theorem 7.3.3 in \cite{godement}) with $E_2^{p,q} = H^p(X, \mathcal{E}xt^q_{\mathcal{A}}(M,N))$ converging to $\Ext^\ast_{\mathcal{A}}(M,N)$. Therefore $\Ext^i_{\mathcal{A}}(M,N)$ is finitely generated over $k$ for all $i$.
\end{proof}

\begin{lemma} \label{ncproplem2}
Suppose $X$ is a mild noncommutative scheme and $M,N \in \mathcal{D}^{\perf}(\mathcal{A})$, then $\Hom_{\mathcal{D}}(\mathcal{A})(M,\Sigma^i N)$ vanishes for all but finitely many $i$.
\end{lemma}

\begin{proof}
By definition of the perfect complexes (see \cite[3.11]{DLM24}), there is a finite open cover $U_i$ of $X$ such that $M_{U_i}$ is quasi-isomorphic to a bounded complex of finite sums of summands of $\mathcal{A}_{U_i}$. So that $M_{U_i} \in \thick_{\mathcal{D}(\mathcal{A}_{U_i})}(\mathcal{A}_{U_i})$. For the same reason as the commutative case,  $\mathcal{E}xt^j_{\mathcal{A}_{U_i}}(\mathcal{A}_{U_i},-)$ vanishes on  $\Qcoh(\mathcal{A})$ for $j \neq 0$. Since $\mathcal{R}\mathcal{H}om_{\mathcal{A}_{U_i}}(-,-)$ is triangulated in each variable, it follows that $\mathcal{R}\mathcal{H}om_{\mathcal{A}_{U_i}}(M_{U_i},N_{U_i})$ is bounded. Since $\mathcal{R}\mathcal{H}om_{\mathcal{A}}(M,N)_{U_i} = \mathcal{R}\mathcal{H}om_{\mathcal{A}_{U_i}}(M_U,N_U)$, and there are finitely many $U_i$, it follows that $\mathcal{R}\mathcal{H}om_{\mathcal{A}}(M,N)$ is bounded. There is a convergent spectral sequence as in \cite[Tag 0BKM]{stacks} with $E_2^{p,q} = H^p(X,H^q(\mathcal{R}\mathcal{H}om_{\mathcal{A}}(M,N))$ converging to $\Hom_{\mathcal{D}(\mathcal{A},X)}(M,\Sigma^\ast N)$. Therefore $\Hom_{\mathcal{D}(\mathcal{A},X)}(M,\Sigma^\ast N)$ is also bounded.
\end{proof}

\begin{prop}\label{ncproper}
Suppose $(\mathcal{A},X)$ is a mild noncommutative scheme with $X$ proper over $k$ and $\mathcal{A} \in \coh(X)$. Then $\mathcal{D}^{\perf}(\mathcal{A})$ is locally finite over $k$.
\end{prop}

\begin{proof}
This follows from Lemmas \ref{ncproplem1} and \ref{ncproplem2}.
\end{proof}

\begin{defn}
An Azumaya algebra over a scheme $X$ is a mild non-commutative scheme $(\mathcal{A},X)$ such that $\mathcal{A}  \in \coh(X)$ and $\mathcal{A}_x$ is a central separable algebra over $\mathcal{O}_{X,x}$ for every $x \in X$.
\end{defn}

\begin{cor}
Suppose $(\mathcal{A},X)$ is an Azumaya algebra over a scheme $X$ which is proper over $k$. Then $\mathcal{D}_{\fd}(\mathcal{D}^{\perf}(\mathcal{A},X)) \simeq \mathcal{D}^b_{\coh}(\mathcal{A},X)$ and $\mathcal{D}^{\perf}(\mathcal{A})$ is finite-reflexive. 
\end{cor}

\begin{proof} 

By Proposition 4.1 of \cite{DLM24} and its proof, $\mathcal{D}_{\Qcoh}(\mathcal{A},X)$ is approximable and by Proposition \ref{ncproper}, $\mathcal{D}^{\perf}(\mathcal{A},X)$ is locally finite. By Proposition 4.2 of \cite{DLM24} $\mathcal{D}^b_{\coh}(\mathcal{A},X) = \mathcal{D}_{\Qcoh}(\mathcal{A},X)^b_c$ and by the proof of Corollary 6.2 in \cite{DLM24b}, there is a $G \in \mathcal{D}^b_{\coh}(\mathcal{A},X)$ such that $\mathcal{D}_{\Qcoh}(\mathcal{A},X) = \overline{\langle G \rangle}_n$. So Theorem \ref{refltest} implies $\mathcal{D}^{\perf}(\mathcal{A},X)$ is finite reflexive. 
\end{proof}

\begin{remark}
By \cite{DLM24},  $\mathcal{D}_{\fd}(\mathcal{D}^{\perf}(\mathcal{A},X)) \simeq \mathcal{D}^b_{\coh}(\mathcal{A},X)$ holds more generally for any Noether scheme.
\end{remark}

\appendix

%
%
%
%
\section{Finite dimensional models via DG-Yoneda}

\label{appendix}

\begin{center}
By Theo Raedschelders and Greg Stevenson
\end{center}

\vspace{3ex}

In \cite{RS20} it was shown that a proper connective DG-algebra over an algebraically closed field is quasi-isomorphic to a finite dimensional DG-algebra (see Corollary~3.12 in particular). This hypothesis is stronger than necessary, but was chosen to simplify things. In this appendix we give a short and direct proof that, over an arbitrary field, such DG-algebras have finite dimensional models. This argument had previously been in private circulation.

Let $A$ be an $A_\infty$-algebra. One can make the category of $A$-modules into a DG-category (we follow the discussion in \cite{ELO} and full details can be found in \cite{LH}). Let us concentrate on the case we care about, namely the DG endomorphism ring of $A$. This is given, as a graded vector space, by
\begin{align*}
\Hom_{\Mod_\infty A}(A,A) = \prod_{m\geq 1} \Hom((\Sigma A)^{\otimes m}, \Sigma A) = \prod_{m\geq 1}\Sigma^{-m+1} \Hom(A^{\otimes m}, A),
\end{align*}
where the undecorated $\Hom$'s are graded vector spaces of graded maps. The differential is given on $\phi = (\phi_l \mid l\geq 1) \in \Hom_{\Mod_\infty A}^n(A,A)$ by
\begin{displaymath}
(d\phi)_l = \sum_{1\leq i \leq l} m_{l-i+1}(1^{\otimes l-i}\otimes \phi_i) - (-1)^n \sum_{r+s+t=l} \phi_{r+1+t}(1^{\otimes r} \otimes m_s \otimes 1^{\otimes t}).
\end{displaymath}
The multiplication is given by
\begin{displaymath}
(\psi \circ \phi)_l = \sum_{1\leq i \leq l} \psi_{l-i+1}(1^{\otimes l-i} \otimes \phi_i).
\end{displaymath}

Now let us suppose that $A$ is minimal and proper, concentrated in the interval $[a,b]$, so $a\leq 0$ and $b\geq 0$. Then $A^{\otimes m}$ is concentrated in the interval $[ma, mb]$ and so $\Hom(A^{\otimes m}, A)$ is concentrated in the interval $[a-mb, b-ma]$. Thus, using the convenient notational abhorrence of replacing a complex by the interval in which it lives, we have
\begin{displaymath}
\Hom_{\Mod_\infty A}(A,A) = \prod_{m\geq 1} [a + m(1-b) -1, b + m(1-a) - 1].
\end{displaymath}
We see that the upper bound approaches infinity and that the lower bound tends to $-\infty$ provided $b\geq 2$. One can easily prove the following two lemmas.

\begin{lemma}\label{lem:cdegrees}
Suppose that $A$ is connective i.e.\ $b=0$. Then the complex $\Hom_{\Mod_\infty A}(A,A)$ is locally finite and concentrated in degrees $\geq a$.
\end{lemma}

\begin{lemma}
If $b\geq 1$ then $\Hom_{\Mod_\infty A}(A,A)$ is not locally finite and if $b\geq 2$ it is unbounded. 
\end{lemma}

In the former case, that is when $A$ is connective, we can just truncate this complex to obtain a good DG-model.

\begin{theorem}\label{thm:cyoneda}
If $A$ is proper and connective then it has a finite dimensional DG-model. In particular, any proper connective DG-algebra has a finite dimensional model.
\end{theorem}
\begin{proof}
By replacing $A$ by a minimal model if necessary, we can assume it has finite dimensional underlying vector space. We have an $A_\infty$-quasi-isomorphism $A\to \Hom_{\Mod_\infty A}(A,A)$. By Lemma~\ref{lem:cdegrees} the DG-algebra $\Hom_{\Mod_\infty A}(A,A)$ is locally finite and bounded below. Given the quasi-isomorphism it is connective and so truncating above we get a finite dimensional DG-algebra which is quasi-isomorphic to $A$. The final statement follows from the first by taking a minimal $A_\infty$-model.
\end{proof}

\bibliographystyle{alpha}
\bibliography{biblio}

\end{document}